\documentclass[center]{paper}

\usepackage[backend=biber,style=nature]{biblatex}
\usepackage[T1]{fontenc}
\usepackage[utf8]{inputenc}
\usepackage{csquotes}
\MakeOuterQuote{"}

\usepackage{amsmath}
\usepackage{amssymb}
\usepackage{amsthm}

\usepackage[stable]{footmisc}
\usepackage{graphicx}
\usepackage{dsfont}
\usepackage{tikz}
\usetikzlibrary{arrows.meta}
\usepackage{subcaption}

\usepackage{hyphenat}
\usepackage{hyperref}
\hypersetup{colorlinks,allcolors=blue}
\usepackage[capitalise]{cleveref}
\crefname{equation}{}{}
\crefformat{footnote}{#2\footnotemark[#1]#3}

\newtheorem{theorem}{Theorem}
\newtheorem*{theorem*}{Theorem}
\newtheorem{definition}{Definition}

\newtheorem{proposition}[theorem]{Proposition}
\newtheorem{corollary}[theorem]{Corollary}

\title{The Augmented Jump Chain}
\subtitle{a sparse representation of time\hyp{}dependent Markov jump processes}
\author{\large Alexander Sikorski  \thanks{\href{mailto:sikorski@zib.de}{sikorski@zib.de}}, Marcus Weber, Christof Schütte
}
\vspace{-1cm}
\institution{Zuse Institute Berlin, November 2020}
\date{\today}

\begin{document}

\maketitle

\begin{abstract} 
Modern methods of simulating molecular systems are based on the mathematical theory of Markov operators with a focus on autonomous equilibrated systems. However, non\hyp{}autonomous physical systems or non\hyp{}autonomous simulation processes are becoming more and more important.
We present a representation of non\hyp{}autonomous Markov jump processes as autonomous Markov chains on space-time.
Augmenting the spatial information of the embedded Markov chain by the temporal information of the associated jump times, we derive the so-called augmented jump chain.
The augmented jump chain inherits the sparseness of the infinitesimal generator of the original process and therefore provides a useful tool for studying time\hyp{}dependent dynamics even in high dimensions.
We furthermore discuss possible generalizations and applications to the computation of committor functions and coherent sets in the non\hyp{}autonomous setting.
After deriving the theoretical foundations we illustrate the concepts with a proof-of-concept Galerkin discretization of the transfer operator of the augmented jump chain applied to simple examples.
\end{abstract}

\section{Introduction}

The last decade of theoretical treatment of simulation methods was characterized by the analysis of autonomous Markov processes. The uniform concept of Markov operators and infinitesimal generators was investigated for these purposes which has led to a rich development of analysis tools in mathematics. In order to be able to benefit from these tools also in the non\hyp{}autonomous case, a broader uniform theoretical framework is required to deal with non\hyp{}autonomous as well as autonomous methods and processes (not only) in molecular simulation.

Physical models often arise from the principle of the cause-and-effect relationship. To think of a process as a sequence of causes and effects straightforwardly leads to the formulation of a Markov process, a process where the future only depends on the present. Thus, due to their flexibility Markov processes have become an important cornerstone for the modeling of many complex systems \cite{karlin2014first,schuttesarich}. Although the dynamical law of the state evolution, e.g. a differential equation with a first-order derivative of the time variable, might be highly nonlinear, the mathematical object that accounts for the transfer of probability densities of system states is a linear transfer operator (the ``adjoint'' continuous counterpart of a transition matrix of a Markov chain). The formulation of Markov processes in terms of transfer operators has proven to be a powerful tool for their analysis. Techniques like transition path theory \cite{tpt}, reaction coordinates \cite{bittracher2018transition,noe2016tica} and coarse graining \cite{schuttesarich}, clustering \cite{roblitz2013fuzzy} and coherent set analysis \cite{froyland2009almost, koltai2016metastability} are just a few methods building on this formalism.

However, the computational cost of these approaches grows with increasing numbers of states and quickly becomes infeasible for high-dimensional problems. Their corresponding formulation in terms of infinitesimal generators or rate matrices \cite{froyland2013estimating,donati2018sqra} promises to alleviate computational costs by making use of the sparse structure in many real world problems, where instantaneous state changes are restricted by a locality assumption. We want to be able to exploit this sparsity also for non-autonomous processes. If one were to find a generator-like object for non\hyp{}autonomous processes, then corresponding methods could be transferred directly.

Physical models mostly refer to self-contained systems that can be isolated in the laboratory and which, therefore, allow for the analysis of autonomous processes, often after equilibration  of the system. However, if we want to study the influence of external forcing (e.g., of external control), transient dynamics or the production rate of catalytic cycles, then non\hyp{}autonomous (i.e. having a time\hyp{}dependent, changing law of state evolution) and non\hyp{}equilibrium systems play an important role.

Whilst there are extensions to the non\hyp{}stationary regimes \cite{koltai2016metastability, koltai2018optimal} we do not know of any such approach inheriting the sparseness of the generator and thus facilitating the analysis of high-dimensional complex systems.

In this article we focus on Markov jump processes which are memoryless stochastic processes continuous in time and discrete in space and have been successfully used in reaction kinetics, queueing theory, Markov state models and network analysis.

Aiming at a sparse approach to non\hyp{}autonomous dynamics we develop a novel representation of time\hyp{}dependent Markov jump processes.
Inspired by recent developments in physics \cite{briggs, schild2019time} which look at time emerging from the order of events rather than as a constantly evolving exogenous entity we look at the process as a series of jumps in space \emph{and} time such that every change that takes place in the system is a change in the spatial {\em and} in the time domain. Formally this amounts to the extension of the ideas of the embedded Markov chain \cite{norris} or semi-Markov processes \cite{limnios2012semi} to the time\hyp{}dependent setting and will lead us to an autonomous process in space-time, the \emph{augmented jump chain}.
A realization of this process, consisting of sequences of space-time points, corresponds to the time\hyp{}continuous trajectory of the original process.
Imagine tracking an ensemble of particles, all starting at the same time in their individual spatial states. We observe their respective jumps which take place in space and time. The transfer operator (of the augmented jump chain) maps the distribution of such an ensemble to the distribution after its next jump whilst retaining the local nature of the original process: particle states still only jump to their ``neighbouring'' states.
Although this operator evolves the classical time in a concurrent manner, we can reconstruct the whole family of Perron-Frobenius operators (for each fixed time) by means of an iterative procedure which we will denote as \emph{synchronization}.
What is more interesting though is that we can compute the action of its dual, the Koopman operator, directly by solving a linear boundary value problem akin to the Chapman-Kolmogorov equation.
This linear problem furthermore resembles the computation of classical committor functions and we show how it naturally leads to an extension of the committor framework to the non\hyp{}autonomous regime with the Koopman operator being a special case of such a non\hyp{}autonomous committor.
We conclude by deriving a (sparse) finite-time Galerkin projection of the transfer operator and applying it to two illustrative examples.

\newcommand{\statespace}{\mathbb{X}}
\newcommand{\timespace}{\mathbb{T}}
\newcommand{\sigmaalg}{\mathcal{A}}
\newcommand{\meas}{\mu}
\newcommand{\dd}{\mathrm{d}}
\newcommand{\PP}{\mathbb{P}}
\newcommand{\ind}{\mathds{1}}
\renewcommand{\SS}{\statespace}
\newcommand{\TT}{\timespace} 
\newcommand{\ST}{\SS \times \TT}
\newcommand{\prop}{\mathcal{P}} 
\newcommand{\koop}{\mathcal{K}}
\newcommand{\JumpOp}{\mathcal{J}}
\newcommand{\MeasOp}{\mathcal{M}}
\newcommand{\Ind}{\mathds{1}}
\newcommand{\JumpAd}{\JumpOp^\dagger}
\section{Background}
In this section we will introduce the notation and recall some basic results needed for the subsequent sections.

Let the set $\statespace = \{x_i\}_{i=1,...,N}$ denote a finite state space and $\{X_t\}_{t\in \TT}$ a time\hyp{}continuous Markov chain (also called Markov jump process) on $\statespace$ with  $\timespace = \mathbb{R}^+_0$ denoting the time domain.
It is well known \cite{schuttesarich} that this process can be described by means of its associated stochastic transition kernel  
\begin{equation}
    k(x, s, y, t) = \PP (X_t = y | X_s = x) \label{eqn:transkernel}
\end{equation}
denoting the conditional transition probabilities.
This kernel gives rise to a family of important transfer operators, 
the \emph{propagator} (or \emph{Perron-Frobenius operator}) acting on densities $\prop: L^1(\statespace) \rightarrow L^1(\statespace), $ 
\begin{equation}
   \left[\prop^{s,t} f\right] (y) = \sum_{x \in \SS}k(x, s, y, t) f(x) 
\end{equation}
and its adjoint, the \emph{Koopman operator} acting on observables $\koop: L^\infty(\statespace) \rightarrow L^\infty(\statespace), $
\begin{equation}
    \left[\koop^{s,t} g\right] (x) = \sum_{y \in \SS} k(x, s, y, t ) g(y).
\end{equation}
These two are adjoint in the sense that $\left<\prop^{s,t}f, g\right> = \left<f, \koop^{s,t}g\right>$ with $\left<\cdot,\cdot\right>$ denoting the corresponding dual pairing.
This equality illustrates that evolving a density $f$ forward in time via $\prop$ and measuring the observable $g$ in the future is the same as pulling the observable $g$ back in time via $\koop$ and applying it to the current state $f$. Therefore the propagator and Koopman operator are also called forward- and backward transfer operator respectively.

Note that we are explicitly interested in time\hyp{}dependent (non\hyp{}autonomous) processes and as such the above objects in general depend on both, the starting time $s$ \emph{and} the end time $t$. In contrast to the time\hyp{}independent (autonomous) regime, where the transfer operators merely depend on the elapsed time $t-s$ and thus form a one-parameter semi-group $\prop^{t-s} := \prop^{s,t}$, the non\hyp{}autonomous pendant does not allow for such a simple construction.

We can nevertheless define the time\hyp{}dependent \emph{infinitesimal generator} at each time $t$ by 
\begin{equation}
    Q(t) = \lim_{\Delta t \searrow 0 } \frac{P_{t}^{t,t+\Delta t} u - u}{\Delta t}.
\end{equation}

We can denote the generator as a matrix $Q(t) = (q_{ij}(t))$ composed of the transition rates from state $x_i$ to $x_j$, 
\begin{equation}
    q_{ij}(t) := \left[ Q(t) \ind_{x_i}\right] (x_j), \quad 1\le i,j \le N,
\end{equation}
with $\Ind$ denoting the indicator function.

We furthermore introduce the shorthand notation for the outbound rate
\begin{equation}\label{eqn:rowsum}
    q_i(t) := - q_{ii}(t) = \sum_{j \neq i} q_{ij}(t), \quad 1\le i \le N
\end{equation}
where the latter equality follows from the fact that our system is (probability-) mass conserving.
In the case of autonomous systems, i.e. $Q(t) \equiv Q$, we will denote these quantities simply by $q_{ij}$ and $q_{i}$.

The generator is of special interest for systems which satisfy the so-called locality assumption, i.e. states only interact with a few other states, as in that case the generator can be represented as a sparse matrix. This diminishes the computational cost in the analysis of many real-world systems, e.g. spatial diffusion processes, where particles can only jump to spatially neighbouring cells. 

The definition of the infinitesimal generator motivates the formal linear equation
\begin{equation}
    \frac{\dd}{\dd t} \prop^{s,t} = Q(t) \prop^{s,t}, \quad \prop^{s,s} = I
\end{equation}
with $I$ denoting the identity operator.

For the autonomous case where $Q(t)\equiv Q$ we indeed know that $\prop^t = e^{tQ}$. This has very useful applications in practice:
Since $\prop^t$ and $Q$ are related via the exponential map, their eigenvectors are the same. Hence they share many statistics, such as their invariant distributions.

One might hope to extend this relationship to the non\hyp{}autonomous regime  by replacing the exponent $tQ$ with its integrated analogue  \cite{reuter2019generalized}
\begin{equation}
    \Omega (t)= \int_{0}^t Q(u) \dd u
\end{equation}
but this does not hold for noncommutative $Q(t)$.
There exist perturbative approaches to the solution of this problem such as the Dyson and Magnus series adjusting for the noncommutativity by computing nested commutators, but these will in general not remain sparse. 
We will tackle the problem from the perspective of the \emph{jump chain} (also called embedded Markov chain) and extend it to the time\hyp{}dependent regime while still inheriting the sparse structure of $Q(t)$.

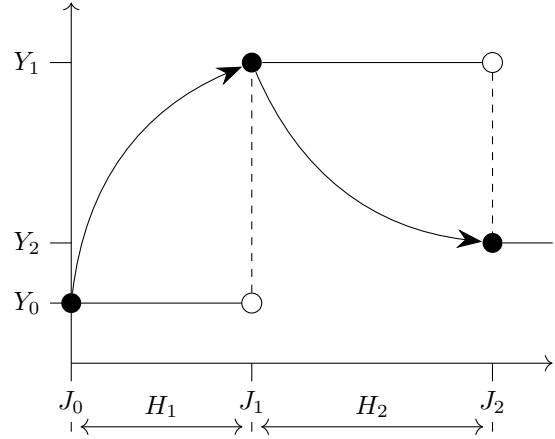
\begin{figure}
    \centering
    \begin{tikzpicture}[scale=.8]
    
    \draw[-{>[scale=1.2]}] (0,0) -- (8,0);
    \draw[-{>[scale=1.2]}] (0,0) -- (0, 6);
    \draw (0,1) node[circle,fill, scale=.8](A){} -- ++ (3,0) node[circle,draw,fill=white, scale=.8](B){};
    \draw[dashed] (B) -- ++ (0,4) node[circle,fill, scale=.8](C){};
    \draw (C) -- ++ (4,0) node[circle,draw,fill=white, scale=.8](D){};
    \draw[dashed] (D) -- ++ (0,-3) node[circle,fill, scale=.8](E){};
    \draw (E) -- ++ (1,0);
    
    \draw[-{Stealth[scale=2]}] (A) to[bend left] (C);
    \draw[-{Stealth[scale=2]}] (C) to[bend right] (E);
    
    \draw (0,0) -- ++(0,-.3) node[below, fill=white](J0){$J_0$};
    \draw (3,0) -- ++(0,-.3) node[below, fill=white](J1){$J_1$};
    \draw (7,0) -- ++(0,-.3) node[below, fill=white](J2){$J_2$};
    
    \draw (0,1) -- ++(-1em, 0) node[left]{$Y_0$};
    \draw (0,5) -- ++(-1em, 0) node[left]{$Y_1$};
    \draw (0,2) -- ++(-1em, 0) node[left]{$Y_2$};
    
    \draw (J0) -- ++(0,-.5) node[pos=.5](H0){};
    \draw (J1) -- ++(0,-0.5) node[pos=.5](H1){};
    \draw (J2) -- ++(0,-0.5) node[pos=.5](H2){};
    
    \draw[{<[scale=1.2]}-{>[scale=1.2]}] (H0) -- (H1) node[pos=.5, above]{$H_1$};
    \draw[{<[scale=1.2]}-{>[scale=1.2]}] (H1) -- (H2) node[pos=.5, above]{$H_2$};
    
    \end{tikzpicture}
    \caption[]%
    {Illustration of the Jump chain
    \par \small
    Depicted (horizontal lines) is a realization of the Markov process $X_t$ in space time. It can be decomposed into its spatial component in form of the jump chain $Y_n$ (dashed lines) and its temporal component, the jump times $J_n$. Gluing both together we end up with the augmented jump chain in space-time (curved arrows).}
\end{figure}

Let us therefore recall the classical construction of the jump chain.
\begin{definition} \label{definition:processes}
Let $X_t$ be a Markov jump process.\\
For $n = 0, 1, 2, ...$ define the \emph{jump times} of $X_t$ to be
\begin{equation}
    J_0 = 0, J_{n+1} = \inf \left\{ t: t > J_n, X_t \ne X_{J_n}  \right\}
\end{equation}
if the infimum is attained or $\infty$ otherwise.
The corresponding \emph{holding times} are defined as 
\begin{equation}
    H_n = J_{n} - J_{n-1}.
\end{equation}
Furthermore define the \emph{jump chain} (also called the \emph{embedded chain}) of $X_t$ to be
\begin{equation}
    Y_n = X_{J_n}.
\end{equation}
\end{definition}

This construction decomposes the original jump process $X_t$ in two components: the temporal component in form of the jump times $J_n$, which amount to the times at which $X_t$ changes its state, as well as the spatial component in form of the jump chain $Y_n$ which keeps track of these states. The holding times, i.e. the differences between the jump times, amount to the time each state remains in the same position.

We can reconstruct the original process from by 
\begin{equation}
\label{eqn:reconstructx}
    X_t = Y_{c(t)}
\end{equation}
with the jump count given by
\begin{equation}
\label{eqn:jumpcount}
    c(t) = \max \{ n \mid J_n \le t\}
\end{equation} 
and $J_n = \sum_{i \le n} H_n$.

The following theorem allows us to characterize both components explicitly in terms of the infinitesimal generator for the case of an autonomous process:
\begin{theorem}\cite[Thm 3.15]{weinan2019applied} \label{thm:aut}Let $X_t$ be an autonomous Markov jump process with infinitesimal generator $Q=(q_{ij})$.

Then the jump chain $Y_n$ is a Markov chain with transition probabilities $\PP(Y_{n+1}=x_j \mid Y_n=x_i) = \tilde q_{ij}$ given by
\begin{align}\begin{split} \label{eqn:jumpprops}
\tilde q_{ij} &= \begin{cases}
q_{ij} / q_i,\phantom{0} &\text{if $j \ne i $ and $q_i \ne 0$}
\\ 0, &\text{if $j \ne i $ and $q_i = 0$}
\end{cases} \\ 
\tilde q_{ii} &= \begin{cases}
0,\phantom{q_{ij} / q_i} &\text{if $q_i > 0$}
\\ 1, &\text{if $q_i = 0$}.
\end{cases}
\end{split}\end{align}

Furthermore the holding times $H_1, H_2, ...$ are independent exponential random variables with parameters $q_{Y_0}, q_{Y_1}, ...,$ respectively.
\end{theorem}

Using this decomposition for sampling, i.e. drawing the next state from the Markov chain $Y_t$ and the exponentially distributed holding time $H_n$ leads to the well known Gillespie (Stochastic Simulation) Algorithm \cite{gillespie1977exact} for sampling from Markov Jump chains.


\section{The augmented jump chain}

In this section we describe the construction of the main object of this study, the \emph{augmented jump chain} for non\hyp{}autonomous processes. Similar to the jump chain of autonomous processes, we decompose the process into its spatial and temporal parts respectively by conditioning either on a specific time or location. Unlike in the autonomous regime however, both parts now explicitly depend on time. By combining both components, i.e. augmenting the spatial with the temporal component, we arrive at an autonomous process on space-time, represented by a new transfer operator, the \emph{jump operator} $\JumpOp$, encoding the original process $X_t$. We then show how to use this operator to reconstruct the classical, non\hyp{}autonomous transfer operators $\prop^{s,t}$, $\koop^{s,t}$ and discuss a more general application for time\hyp{}dependent committors.

\subsection{Construction}

\begin{definition}
Define the \emph{augmented jump chain} to be the tuple
\begin{equation}
\label{eqn:ajc}
    (Y,J)_n = (Y_n, J_n)_n \quad \text{for } n=0,1,2,...
\end{equation}
where the jump chain and jump times are defined as in  \cref{definition:processes}.
\end{definition}

We call this the \emph{augmented} jump chain since its state space is that of the original process $X_t$ (or its jump chain $Y_n$) augmented by the time component.
Note however that unlike in classical augmentation schemes (e.g. the augmentation of non\hyp{}autonomous differential equation) the  ``internal'' time component $J_n$ does not evolve linearly with the ``external'' time $n$ of the augmented jump chain.

The augmented jump chain now gives us a tool to analyse the time\hyp{}continuous spatially-discrete Markov process $X_t$ by means of a discrete-time Markov chain $(Y,J)_n$ on the product space $\SS$, i.e. to look at the process on a per-jump basis.
Analogue to the autonomous case we can transfer forth and back between the two representations, either by the definition of the augmented Markov chain \cref{eqn:ajc} or the evaluation of the jump chain \cref{eqn:reconstructx} at the time\hyp{}corresponding jump counts \cref{eqn:jumpcount}.

Due to the time dependent structure of the process $X_t$ the transition rules change compared to the autonomous case (\cref{thm:aut}):

\begin{theorem}
\label{thm:ajc}
The \emph{augmented jump chain} $(Y,J)_n$ is a time\hyp{}homogeneous\slash{}autonomous Markov chain on $\ST$ with transition kernel
\begin{equation}
\label{eqn:jumpkernel}
    k(x_i, s, x_j, t) = \tilde q_{ij}(t) q_i(t) \exp\left({-\int_{s}^{t} q_i(u) \dd u}\right)
\end{equation}
for $s<t$ or $k=0$ otherwise with $\tilde q_{ij}(t)$ being defined as the time\hyp{}dependent equivalents of eq. \cref{eqn:jumpprops}. \\
The corresponding transfer operator is given by the \emph{jump operator} $\JumpOp: L^1(\ST) \rightarrow L^1(\ST) $
\begin{equation}
    \left[\JumpOp\rho\right] (y,t)= \int_T \sum_{x \in \SS} k(x, s, y, t) \rho (x, s) \mathrm{d}s
\end{equation}
and its \emph{adjoint} $\JumpAd: L^\infty (\ST) \rightarrow L^\infty (\ST) $ by
\begin{equation}
    \left[\JumpAd\rho\right] (x,s)= \int_T \sum_{y \in \SS} k(x, s, y, t) \rho (y, t) \mathrm{d}t.
\end{equation}

\end{theorem}
\begin{proof}
Since $J_{n+1} > J_n$ by definition we have $k = 0$ for $s\ge t$. Let us therefore consider the case of and $s < t$.

Since $X_t$ is Markovian, the jump location at a specific jump time depends solely on the generator at that time, so similar to the autonomous case we have
\begin{equation}
    \PP \left( Y_{n+1} = x_j \mid Y_n = x_i, J_{n+1} = t \right) = \tilde q_{ij}(t)
\end{equation}

Unlike in the autonomous case the jump times now depend on the time\hyp{}dependent rates. We therefore replace the homogeneous exponential distribution with its non\hyp{}homogeneous complement, which is also known as the risk of mortality\slash{}hazard function (c.f. appendix):
\begin{equation}
    \PP \left (J_{n+1} = t \mid J_n = s, Y_n = x_i \right) = q_i(t) \exp \left(- \int_s^t q_i(u) \dd u\right)
\end{equation}

Putting these together, we end up with the desired result
\begin{equation}
\begin{aligned}
 & k(x_i, s, x_j, t) \\
     &= \PP (Y_{n+1}=x_j, J_{n+1} = t \mid Y_n = x_i, J_n = s) \\
    & = \PP(Y_{n+1}=x_j\mid J_{n+1} = t, Y_n = x_i, J_n=s) 
     \PP(J_{n+1}=t\mid Y_n = x_i, J_n = s) \\
     &= \tilde q_{ij}(t) q_i(t) \exp{\left(-\int_{s}^{t} q_i(t) \dd u\right)}. 
\end{aligned}
\end{equation}

\end{proof}

The given theorem gives allows us to sample realizations of the augmented jump chain by successively generating samples from the probability density \begin{equation}
    (Y_{n+1}, J_{n+1}) \sim k(Y_n, J_n, Y_{n+1}, J_{n+1})
\end{equation} by drawing the jump time from the inhomogeneous exponential distribution followed by the jump location from the embedded Markov chain at that time.
This procedure for sampling from time\hyp{}dependent Markov Jump processes is also known as the temporal Gillespie algorithm \cite{vestergaard2015temporal}.

Having the transition kernel it is natural to look at the associated transfer operators which in this case evolve space-time densities. In the following subsections we will show how they enable us to reconstruct the transfer operators $\koop, \prop$ of the original process $X_t$.

Let us denote all space-time distributions $\rho \in L^1(\ST) $ which have all their mass at a single time\hyp{}slice $t_0$ as \emph{spacelike}.
Given some spacelike initial distribution $\rho$ for the augmented jump chain $(Y_0, J_0) \sim \rho$ its subsequent space-time states are distributed according to
$$ (Y_{n}, J_{n}) \sim \JumpOp^n \rho.$$

\subsection{Reconstruction of the Propagator}
\label{sec:prop}

The application of the jump operator $\JumpOp$ to a spacelike initial density $\rho$ returns the density of the locations of its next jump events in space-time. Whilst the initial density's location in time was fixed by construction, its image under $\JumpOp$, i.e. the location of the next jump, is spread out in time; one may regard the result as desynchronized.
This leads to the question what can be said about the distribution at a a future fixed time\hyp{}slice $\SS \times \{t\}$.
Starting from the jump-activity, the superposition of \emph{all} subsequent jumps, and accounting for the probability to remain in place (i.e. not jump) until the target time we return to the synchronized view by reconstructing the classical propagator $\prop$ from the augmented jump chain.

\begin{definition}
The jump-activity $E: \ST \rightarrow \SS$ is given by 
\begin{equation}
\label{eqn:jumpact}
    Ef := \sum_{n=0}^\infty \JumpOp^n f.
\end{equation}
\end{definition}

Starting with a spacelike distribution $f$, the corresponding jump-activity $Ef$ is the density of all induced jump events, similar to the activity of a Geiger-counter over time. In the general case $Ef$ can be interpreted as the density of jumps induced by a superposition of spacelike distributions.

Note that $E$ admits the form of a Neumann-series, i.e. $E = (Id - \JumpOp)^{-1}$.

\begin{definition}
Define the survival probability from time $t_0$ to time $t_1$ at point $x_i\in \SS$ as 
\begin{equation} 
S(x_i, s, t) :=  \PP [J_{n+1} > t | Y_n = x_i, J_n = s] = \exp\left(-\int_{s}^{t} q_i(u) \dd u\right).
\end{equation}
Define the \emph{synchronization operator} $S^t: L_1(\ST) \rightarrow L_1(\SS)$ at time $t$ by:
\begin{equation}
\label{eqn:sync}
    \mathcal{S}^t f(y) = \int_{s \le t} f(y,s) S(y,s,t) \dd s
\end{equation}
\end{definition}
The synchronization operator takes a space-time density and projects it onto a specific time by weighting each point with its probability to survive until that time. Starting from a space-like density we are now in the position of constructing all consequent jumps and synchronizing them to a specific time, thereby reconstructing the action of the classical propagator Perron-Frobenius operator:

\begin{theorem}
Let $\bar{f} \in L_1(\SS)$ and $f(x,t) = \delta(t) \bar f(x) \in L_1(\ST)$ its spacelike embedding. The measurement operator $\MeasOp^t = \mathcal{S}^t E$ reconstructs the action of the classical propagator $\prop$, i.e.
\begin{equation}
\MeasOp^t f = \mathcal{S}^t E f= \prop^{0,t} \bar f
\end{equation}
\end{theorem}

\begin{proof}

The probability to be in point $x$ at time $t$ is equal to the sum of the probabilities to jump to $x$ just before time $t$ for every jump time $n$:

\begin{equation}
\begin{aligned}
\left(\prop^{0,t}\bar f\right)(x) &= \PP[X_t = x \mid X_0 \sim \bar f] \\
 &= \sum_{n=0}^{\infty} \PP[Y_n = x, J_n \le t < J_{n+1} \mid Y_0 \sim \bar f, J_0 = 0] \\
\end{aligned}
\end{equation}
which can be further decomposed to jumping to $s$ and staying there
\begin{equation}
\begin{aligned}
\left(\prop^{0,t}\bar f\right)(x)
 &= \sum_{n=0}^{\infty} \int_{s\le t} \bigg( \PP\left[Y_n = x, J_n = s \mid Y_0 \sim \bar f , J_0 = 0\right] \cdot  \\
 &\phantom{= \sum_{n=0}^{\infty} \int_{s\le t} \bigg(, } \PP\left[J_{n+1} > t\mid Y_n = x, J_n = s\right] \bigg) \dd s\\
 &= \sum_{n=0}^{\infty} \int_{s\le t} \JumpOp^n  f (x,s) S(x,s,t) \dd s\\
 &= \int_{s\le t}  E f (x,s) S(x,s,t) \dd s \\
 &= \mathcal{S}^t E f
\end{aligned}
\end{equation}

\end{proof}

\subsection{Reconstruction of the Koopman operator}
Instead of solving the propagator directly by computing all possible jumps, as done in the section above, we can solve for the transition kernel of the process $X_t$ with a single jump.
Similar to the Kolmogorov backward equation we will transport the transition kernel $k(x,s,y,t)$ for fixed $y,t$ backwards in time.
This enables us to obtain the propagator by solving a family of boundary value problems.
Furthermore we can compute its adjoint, the Koopman operator, by solving just a single boundary value problem (BVP).

\newcommand{\f}{f^{y,t}}
\newcommand{\ff}{\f(x,s)}
\begin{theorem}
Let

\begin{equation}
    f^{y,t}(x,s) := \PP (X_t = y \mid X_s = x).
\end{equation}
Then $f^{y,t}$ satisfies the inhomogeneous linear boundary value problem 
\begin{equation}
\label{eqn:kolmogorov}
\begin{alignedat}{3}
f^{y,t}(x,s) &= \JumpAd f^{y,t}(x,s) + S(x,s,t) \delta_{xy}, \quad && \text{for }s<t\\
f^{y,t}(x,s) &= \delta_{xy}, && \text{for } s=t. 
\end{alignedat}
\end{equation}
with $\delta_{xy}$ denoting the Kronecker delta.
\end{theorem}
\begin{proof}
Define 
\begin{equation}
    c(t) = \max \{n \mid J_n \le t\}
\end{equation}
to be the last index of the jump chain before crossing time $t$.
Using the law of total probability we see that we can decompose the probability $\f$ into the cases of either jumping or staying 
\begin{equation}
\begin{alignedat}{2}
    \ff &=  \PP(X_t = y \mid X_s = x) =  \PP(Y_{c(t)} = y \mid Y_0 = x, J_0 = s) \\
    & = \PP(Y_{c(t)} = y, c(t) = 0 \mid Y_0 = x, J_0 = s) \\
    & + \PP(Y_{c(t)}=y, c(t) > 0 \mid Y_0 = x, J_0 = s) \\
\end{alignedat}
\end{equation}
The first part reduces to
\begin{equation}
\begin{aligned}
&\PP(Y_{c(t)} = y, c(t) = 0 \mid Y_0 = x, J_0 = s) \\
     &= \PP(c(t) > 0 \mid Y_0 = x, J_0 = s) \, \PP (Y_{c(t)} \mid Y_0 = x, J_0 = s, c(t) = 0) \\
        &= S(x,s,t) \delta_{xy} 
\end{aligned}
\end{equation}
For the second part, since $c(t)>0$, we can decompose the jump event as
\begin{equation}
\begin{aligned}
    &\PP(Y_{c(t)}=y , c(t) > 0 \mid Y_0 = x, J_0 = s) \\
    =  &\int \PP (Y_{c(t)} = y \mid Y_1 =z, J_1 = u) \PP(Y_1 = z, J_1 = u \mid Y_0 = x, J_0 = s) \dd z \dd u \\
    =  &\int \PP (Y_{c(t)} = y \mid Y_1 =z, J_1 = u) k(x,s,z,u) \dd z \dd u \\
    =  & \JumpAd \f(x,s)
\end{aligned}
\end{equation}
where the last equality follows from
\begin{equation}
    \begin{aligned}
        \f(z,u) &= \PP(Y_{c(t)} = y \mid Y_{0} = z, J_{0} = u) \\
        &= \PP(Y_{c(t)} = y \mid Y_{1} = z, J_{1} = u)
    \end{aligned}
\end{equation}
which holds due to $(Y,J)$ being homogeneous.

Putting it all together and treating the special case of $s=t$ implying $c(t)=0$ we arrive at the stated boundary value problem.
\end{proof}

Since $\f$ is just the transition kernel \labelcref{eqn:transkernel} of the original process for fixed $(y,t)$, i.e. 
\begin{equation}
    \ff = k(x,s,y,t)
\end{equation}
we can represent the propagators $\prop$ and the Koopman operators $\koop$ in terms of $\f$ as
\begin{equation}
\begin{aligned}
    \prop^{s,t} g(y) &= \int \ff g(x) \dd x, \\
    \koop^{s,t} g(x) &= \int \ff g(y) \dd y. 
\end{aligned}  
\end{equation}

Note that the evaluation of the propagator requires the solution of the BVP \labelcref{eqn:kolmogorov} for each $y$, which corresponds to solving for the fundamental matrix of the system.
The evaluation of the Koopman operator on the other hand can be computed by solving a single BVP:

\begin{corollary}
\label{cor:koop}
The evaluation of the Koopman operator $K(x,s) = \koop^{s,t}g(x)$ 
satisfies the inhomogeneous linear boundary value problem
\begin{align}
\label{eqn:bvp}
\begin{alignedat}{3}
K(x,s)&= \JumpAd K(x,s) + S(x,s,t) g(x), \quad && \text{for }s<t\\
K(x,t)&= g(x), && \text{for } s=t.
\end{alignedat}
\end{align}
\end{corollary}

\begin{proof}
This follows immediately from
\begin{equation}
    K(x,s) = \int f^{y, t}(x,s) g(y) \dd y
\end{equation}
by integration of the product of BVP \labelcref{eqn:kolmogorov} and $g$ over $y$ and the linearity of $\JumpAd$.
\end{proof}

\subsection[Connections to committor functions]{Connections to committor functions\footnotemark}

\footnotetext{\label{foot}These section are not special to the augmented jump chain but work similarly in the classical time\hyp{}augmented setting, albeit the jump chain may allow for sparse formulations (see \cref{sec:sparsity}).}

Formally the approach above is very similar to the computation of \emph{committor functions} $c(x)$ giving the probability to hit some set $A$ before some other set $B$ conditioned on starting in $x$.
Classically the stationary committor function for the sets $A,B \subset \SS$ is the function $c: \SS \rightarrow [0,1]$ satisfying the boundary value problem 
\begin{equation}
    c = \koop^t c, \quad \text{in } \SS \setminus (A \cup B)
\end{equation}
with prescribed boundary values $\left.c\right|_A\equiv 1$ and $\left.c\right|_0 \equiv 1$ \cite{schuttesarich}.
This approach was recently extended to non\hyp{}autonomous dynamics for finite-time and periodic systems \cite{helfmann2020extending}.
Generalizing furthermore to time\hyp{}dependent target sets it may be useful to think of \emph{committor functions on space-time}.

Indeed the Koopman operator applied to an indicator function of some set $G \subset \SS$ can then be interpreted as such a generalized committor function $K(x,s)$, i.e. the probability to hit space-time set $A = G \times \{t\}$ before $ B = \SS \backslash G \times \{t\}$ (see sketch 3 of \cref{fig:committor}):
\begin{equation}
    K(x,s) = \koop^{s,t} \Ind_G.
\end{equation}

By generalizing the BVP \cref{eqn:bvp} to a wider class of boundary values, we may be able to compute such \emph{non\hyp{}autonomous committors}, i.e. committors of non\hyp{}autonomous systems with time\hyp{}dependent target sets.
Solutions to these equations will still satisfy the correct propagation of probability according to the law of the process. 
Choosing appropriate space-time boundary sets $A$ and $B$ may then allow to compute many interesting quantities such as the stationary committor, finite-time hitting probabilities or arbitrary space-time committors by solving the corresponding linear problem (\cref{fig:committor}).

\begin{figure}
    \centering
\begin{tikzpicture}
\draw [<->,thick] (0,2) node (yaxis) [above] {$\SS$}
        |- (2,0) node (xaxis) [right] {$\TT$};
\draw 
        (0,.7) -- ++(1.5,0) edge[dashed] ++(.5,0) node[below,pos=.5] {A}
        (0,.2) -- ++(1.5,0) edge[dashed] ++(.5,0);
\draw 
        (0,1.7) -- ++(1.5,0) edge[dashed] ++(.5,0) node[below,pos=.5] {B}
        (0,1.2) -- ++(1.5,0) edge[dashed] ++(.5,0);

\begin{scope}[shift={(3,0)}]
    \draw [<->,thick] (0,2) node (yaxis) [above] {$\SS$}
        |- (2,0) node (xaxis) [right] {$\TT$};
    \draw 
            (0,.7) -- ++(1.5,0) node[below,pos=.5] {A}
            (0,.2) -- ++(1.5,0);
    \draw 
            (1.5,0) -- ++(0,1.5) edge[dashed] ++(0,.5) node[right,pos=.7] {B};
\end{scope}
  
\begin{scope}[shift={(6,0)}]
    \draw [<->,thick] (0,2) node (yaxis) [above] {$\SS$}
        |- (2,0) node (xaxis) [right] {$\TT$};
    \draw 
            (1.5,.7) -- ++(.5,0) node[below,pos=.5] {A}
            -- ++(0,-.5) -- ++(-.5,0);
    \draw 
            (1.5,0) -- ++(0,1.5) edge[dashed] ++(0,.5) node[right,pos=.7] {B};
\end{scope}
  
\begin{scope}[shift={(9,0)}]
    \draw [<->,thick] (0,2) node (yaxis) [above] {$\SS$}
        |- (2,0) node (xaxis) [right] {$\TT$};
    \draw (.7,1.5) circle [x radius=.6, y radius=.3]
            node {A};
    \draw (.7,.5) circle [x radius=.6, y radius=.3]
            node {B};
    \draw 
            (1.5,0) -- ++(0,1.5) edge[dashed] ++(0,.5) node[right,pos=.7] {B};
\end{scope}

\end{tikzpicture}
    \caption[]%
    {Sketches of sets for space-time committors \par \small
    By choosing suitable space-time sets A and B we can construct different interesting committor-like objects.
    From left to right: The classical (stationary) committor, a finite-time hitting probability, fixed time hitting probability (Koopman operator), (fully) non\hyp{}autonomous committor.}
    \label{fig:committor}
\end{figure}
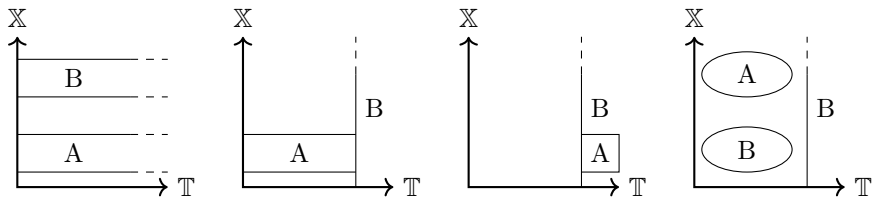

\subsection[Connections to coherence]{Connections to coherence\cref{foot}}
In the context of stationary Markov processes, metastabilities, that is regions of space $A \subset \SS$ which are almost-invariant under time\hyp{}evolution, 
\begin{equation}
    \koop \Ind_A \approx \Ind_A
\end{equation}
have proven to be a very useful notion for gaining understanding as well as dimensionality reduction of the system.

Extending this approach to the time\hyp{}dependent regime the analogue to metastability is given by coherence \cite{koltai2016metastability}.
A set $A \subset \SS$ is \emph{forward-backward coherent} if there exists a set $B \subset \SS$ such that
\begin{equation}
\koop^{s,t} \Ind_A \approx \Ind_B \text{ and } \koop_-^{t,s} \Ind_B \approx \Ind_A
\end{equation}
where $\koop_-^{t,s}$ is the appropriately defined Koopman operator of the backward process. This definition asserts that $A$ stays "coherent" under time\hyp{}evolution from $s$ to $t$ in the sense that the space-regions $A$ and $B$ at times $s$ resp. $t$ have an almost-certain one to one correspondence.
Note that forward-backward coherence also implies that (almost) no mass in set $B$ came from outside of set $A$.

The augmented jump chain naturally gives rise to a further possible notion of coherence in terms of almost-invariant space-time regions:
\begin{equation}
    \JumpOp^\dagger \Ind_C \approx \Ind_C, \quad C\subset \ST
\end{equation}
Whilst this only implies what we would call forward coherence this notion may suffice for many applications and a similar construction involving a backward operator to study forward-backward coherence should pose no difficulties.

Moreover we can formally introduce a probabilistic notion of coherence in the form of coherent functions:
\begin{definition}
Let $f: \ST \rightarrow [0,1]$. We call f a \emph{forward coherent function} if it satisfies
\begin{equation}
\label{eqn:cohfun}
    \JumpOp^\dagger f \ge f.
\end{equation}
\end{definition}
The coherent function $f$ allows for the interpretation as a probability density for a space-time region belonging to (observing) the coherent regime described by $f$. If a point has high density, i.e. probably belongs to the coherent regime, this probably will not decrease with the temporal evolution, i.e. it will likely stay in that coherent regime.

We easily see that these functions are not unique by adding a probability in "the future", e.g. $f'(x,s) = f(x,s)$ if $s<T$ and $f'(x,s) = 1$ otherwise.
This however weakens the notion of the corresponding coherent regime, since from time $T$ anything belongs to it.
So there is a whole family of coherent functions and depending on the context they may allow to model many requirements leading to optimization problems such as for example finding the spatially "most concentrated" coherent function losing the least amount of mass per time or the "most certain" function coming from some source and hitting some target region in space-time and many more.

Moreover due to its integral approach of time the augmented jump chain allows not only for coherence with respect to fixed starting- and end times but may allow to find coherent regimes for the intrinsic time scales of the process. 
Finally it might be interesting to decompose the space-time into coherent regimes to obtain a coarse-grained description of the system.

\newpage{}

\section{Numerical discretization}
The jump operator acts on the space-time $\ST$ which due to the continuity of time is an infinite space. In order to allow for numerical computations we will discretize the space-time $\ST$ and the jump operator $\JumpOp$. In the case of spatially sparse generators their sparsity will carry over to the matrix representation of $\JumpOp$.

A straightforward approach would be to discretize time into $M$ intervals $T_l := [t_{i-1}, t_i)$. One could then compute the transition probabilities.
\begin{equation}
    \PP(Y_{n+1} = x_j, J_{n+1} \in T_l \mid Y_{n} = x_i, J_n = t_k).
\end{equation}
Note however that we had to assume a fixed starting point $(t_k)$, since we have lost the information about the distribution inside an interval.
One can interpret this as shifting all the particles that jump into a time\hyp{}interval to the beginning of that interval. In order to compensate for that error we will work with an Galerkin discretization onto indicator functions of these intervals (also called Ulam discretization):

\subsection{Ulam-Galerkin projection}
\label{sec:galerkin}

\begin{definition}
Partition the finite time\hyp{}interval  $[0, T]$ into $M$ disjoint intervals $T_k := (t_{k-1}, t_k]$ of size $\Delta T_k = |T_k|$, with $t_0 = 0, t_M = T$.
Define $\hat \JumpOp: L^2(\mathbb{U}) \rightarrow L^2(\mathbb{U})$ to be the Galerkin projection of $\JumpOp$ onto $\mathbb{U} = \text{span}\left\{ \Ind_{il} \right\}_{1\le i \le N ; \, 1 \le l \le M}$:
\begin{equation}
    \hat \JumpOp_{ikjl} := \frac{\left< \Ind_{jl}, \JumpOp \Ind_{ik} \right>}{\left< \Ind_{ik}, \Ind_{ik} \right>}
\end{equation} 
where $i,j \in \{1,...,N\}$, $k,l\in\{1,...,M\}$ and 
\begin{equation}
    \Ind_{ik}(x,s) = \begin{cases} 1 \quad \text{if } x = x_i, s \in T_k \\
        0 \quad \text{else}
    \end{cases} 
\end{equation}
\end{definition}

These entries correspond to the assumption of a uniform prior $\mathcal{U}$ for the starting time of the particles inside the intervals:

\begin{equation}
    \hat \JumpOp_{ikjl} = \PP (Y_{n+1} = x_j, J_{n+1} \in T_l \mid Y_n = x_i, J_n \sim \mathcal{U}(t_k))
\end{equation}

The following proposition shows how to compute the entries assuming a finite time horizon and a generator which is piecewise constant on each time interval:

\begin{proposition}\label{thm:Tconstruct}
Assume the generator $Q(t)$ is constant on each $T_k$.
We then have 
\begin{equation}
\begin{aligned}
\hat \JumpOp_{ikjl} = \begin{cases}
    \Delta T_k^{-1} \tilde q_{ij}(t_l) q_i(t_k)^{-1} (1-s_{ik}) (1-s_{il})
        \prod\limits_{k<m<l} s_{im}
            &\text{if $k<l$} \\
    \Delta T_k^{-1} \tilde q_{ij}(t_k) q_i(t_k)^{-1} \left( s_{ik} + \Delta T_k q_i(t_k) - 1\right)
            &\text{if $k=l$} \\
    0 & \text{else}
    \end{cases}
\end{aligned}
\end{equation}
where $s_{ik} := \exp\left({-\Delta T_k q_i(t_k)}\right)$

\end{proposition}

\begin{proof}
We have
\begin{equation}
\begin{aligned} 
    \hat \JumpOp_{ikjl} &= \frac{\left< \Ind_{jl}, \JumpOp \Ind_{ik} \right>}{\left< \Ind_{ik}, \Ind_{ik} \right>} \\
     &= \Delta T_k^{-1} \int_{T_l} \int_{T_k} k(x_i, \tau_0, x_j, \tau_1) d\tau_0 d\tau_1 \\
     &= \Delta T_k^{-1} \int_{T_l} \int_{T_k} \tilde q_{ij}(\tau_1) q_i(\tau_1) \exp \left(-\int_{\tau_0}^{\tau_1} q_i(\tau) d\tau \right) d\tau_0 d\tau_1
\end{aligned}
\end{equation}

For $k < l$ we decompose the integral in the exponent on the time intervals

\begin{equation}
\label{eqn:qintegral}
\begin{aligned}
   \int_{\tau_0}^{\tau_1} q_i(\tau) d\tau &= \int_{\tau_0}^{t_{k}} q_i(\tau) d\tau+ \sum_{k<m<l} \int_{t_{m-1}}^{t_m} q_i(\tau) d\tau+ \int_{t_{l-1}}^{\tau_1} q_i(\tau) d\tau \\
   &= (t_{k} - \tau_0) q_i(t_k) + (\tau_1 - t_{l-1}) q_i(t_l)+ \sum_{k<m<l} \Delta T_m q_i(t_m) 
\end{aligned}
\end{equation}

Furthermore computing

\begin{equation}
\begin{aligned}
    \int_{T_k} \exp(-(t_{k}-\tau_0)q_i(t_k)) d\tau_0 = \int_0^{\Delta T_k} \hspace{-2em}\exp(-\tau q_i(t_k)) d\tau = (1-\exp\left(-\Delta T_k q_i(t_k)\right))q_i(t_k)^{-1}
\end{aligned}
\end{equation}

and similarly for the $\int_{T_l}$ part leads us to
\begin{equation}
\begin{aligned}
   \hat \JumpOp_{ikjl} = \Delta T_k^{-1} \tilde q_{ij}(t_l) q_i(t_l) \exp \left(- \hspace{-.3em}\sum_{k<m<l} \hspace{-.3em} \Delta T_m q_i(t_m) \right) 
    (1-s_{ik})q_i(t_k)^{-1} (1-s_{il})q_i(t_l)^{-1}
\end{aligned}
\end{equation}

In the case of $k = l$ we have to take care that the arrival time must be larger than the initial time ($t_0>t_1$ implies $k(x_i,t_0,x_j,t_1) = 0$) and we hence compute
\begin{equation}
\begin{aligned}
    \hat \JumpOp_{ikjl} &= \Delta T_k^{-1} \tilde q_{ij}(t_k) q_i(t_k) \int_{t_{k-1}}^{t_{k}} \int_{\tau_0}^{t_{k}} \exp\left(-\int_{\tau_0}^{\tau_1} q_i(t_k) d\tau\right) d \tau_1 d\tau_0 \\
    &=\Delta T_k^{-1} \tilde q_{ij}(t_k) q_i(t_k) \left( s_{ik} + \Delta T_k q_i(t_k) - 1\right) q_i(t_k)^{-2}
\end{aligned}
\end{equation}

For $k>l$ it follows that $\hat \JumpOp_{ikjl} = 0$.
\end{proof}

Using a space\hyp{}major indexing scheme we can rearrange the discretization to a matrix $J = (J_{ab})_{a,b \in \{1,...,NM\}}$ via 
\begin{equation}
    J_{i + (k-1)M, j + (l-1)M} := \hat \JumpOp_{ikjl}
\end{equation} 
as illustrated in \cref{fig:matrix}.
Since the Galerkin projection of the adjoint is the transpose of the Galerkin projection the matrix $J$ corresponds to $\JumpOp$ as well as $\JumpOp^\dagger$ when applying the vectors from either the left resp. the right side.

 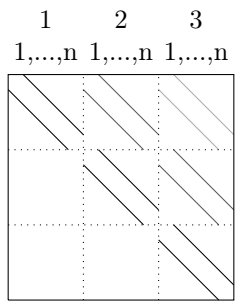
\begin{figure}
    \centering
    \begin{tikzpicture}[scale=1]
        \newcommand\os{1/5}
        \draw (0,0) -- (3,0) -- (3,3) -- (0,3) -- (0,0);
        \draw[dotted] (0,1) -- (3,1);
        \draw[dotted] (0,2) -- (3,2);
        \draw[dotted] (1,0) -- (1,3);
        \draw[dotted] (2,0) -- (2,3);
        \foreach \i in {0,1,2}
            {
            \node [above] at (1/2+\i,3) {1,...,n};
            
            \foreach \j in {1,2,3}
            {
                \ifnum \j>\i
                    \draw[opacity=(4-\j+\i)/3] (\j - 1 + \os, -\i + 3) -- ++(1-\os,-1 + \os);
                    \draw[opacity=(4-\j+\i)/3] (\j - 1, -\i + 3 - \os) -- ++(1-\os,-1 + \os);
                \fi
            }
            }
        \node [above] at (1/2,3.5) {1};
        \node [above] at (3/2,3.5) {2};
        \node [above] at (5/2,3.5) {3};
    \end{tikzpicture}
    \caption{Matrix representation of the Galerkin discretization in space-major order (the outer indices denote the time and the inner ones the space).
    A horizontal line of the matrix represents the probabilities to jump to a space-time point when starting at a fixed space-time position. The probabilities are decreasing (non\hyp{}homogeneous\hyp{}) exponentially with the time blocks. The sparsity structure in each time\hyp{}block corresponds to that of the generator at that time. We have a tridiagonal block structure since particles only move forward in time.}
    \label{fig:matrix}
\end{figure}

We would like to note that this is a very crude proof-of-concept discretization providing the means to compute above posed problems numerically. 
The assumption of piecewise constant inhomogeneity $Q(t)$ may be dropped when solving the corresponding integrals \cref{eqn:qintegral} either analytically or by quadrature.
In the case of varying implicit timescales $0< q_i\ll q_j$ we expect adaptive time\hyp{}discretizations to be of aid.
Since the survival times are exponentially decaying a cutoff may reduce complexity for long time\hyp{}horizon calculations.
As always with Galerkin methods one can adapt this method with different ansatz functions \cite{vikram}.
Although these are import questions the discretization is not the focus of this manuscript and we defer them for later research.

\subsection{Sparseness and Complexity}
\label{sec:sparsity}
We constructed the augmented jump chain with the goal of sparsity in mind.
We can see that the transition kernel \cref{eqn:jumpkernel} of the jump chain is given in terms of the rates $\tilde{q_{ij}}(t)$.
Therefore the sparsity of the infinitesimal generator, a property very common in many applications, is inherited by this representation.
This concept is also reflected in our discretization: Whenever $q_{ij}(t_l)$ is zero, $\hat \JumpOp_{ikjl}$ and the corresponding entry in the matrix $J$ is zero as well.

Whilst the matrix $J$ is much bigger ($NM \times NM$) than e.g. the generator of an autonomous system ($N \times N$), some increase in complexity is to be expected when going from the non\hyp{}autonomous to the autonomous regime.
We hence might compare our approach to the classical augmentation of the transfer operator. The classical augmentation leads to a band diagonal block matrix where the first off diagonal blocks are composed of the transition matrices between the individual time points $t_k$. Whilst the number of non-zero blocks, $\mathcal{O}(M)$, is much smaller then in our suggested approach, $\mathcal{O}(M^2)$, each of these blocks is dense.

This difference becomes crucial when considering very big, sparse systems, such as e.g. diffusion or molecular dynamics on high-dimensional spaces: Using a regular grid of $L$ subdivisions in each of the $D$ space dimensions we end up with $N=L^D$ Markov states. However since each of those only interacts with its respective neighbours the generator has only $2LD$ nonzero entries. Therefore, whilst the augmented transition matrix has $ML^D$ non-zero entries, the augmented jump chain matrix $J$ has $\mathcal{O}(M^2 LD)$ entries, thus practically eliminating the exponential curse of dimensionality.

\section{Numerical examples}
In this section we will first illustrate the developed concepts on a simple time-dependent 2-state model and then compute basic error statistics for the jump operator discretization of the overdamped Langevin dynamics in a 2-dimensional potential landscape.

\subsection{A simple 2-state model}

For the first example we consider two states, $\SS = \{A, B\} $ on the time interval $\TT=[0,8]$. The dynamics of the jump process at each time is fully determined by the respective rates of transitions from $A \rightarrow B$ and $B \rightarrow A$ respectively.
Aiming for a non-autonomous but simplistic example we define the process to consist of two phases. In the first half of the time interval it is possible to transition from $A$ to $B$ at rate 1 whereas $B$ is absorbing and in the second half we reverse the roles:
$$ Q(t) = \begin{pmatrix} -\Ind_{t<4} & \phantom{-} \Ind_{t<4} \\ \phantom{-} \Ind_{t\ge4} &  - \Ind_{t\ge4} \end{pmatrix}
$$
We then compute the Galerkin discretization of the jump operator as in section \ref{sec:galerkin}.
Partitioning the time interval into $M=8$ uniform intervals we obtain the jump matrix $J$ depicted in the left of Figure \ref{fig:ex1}.
\begin{figure}[]
    \centering
    \begin{subfigure}{0.49\textwidth}
                  \centering
                  \includegraphics[width=\textwidth]{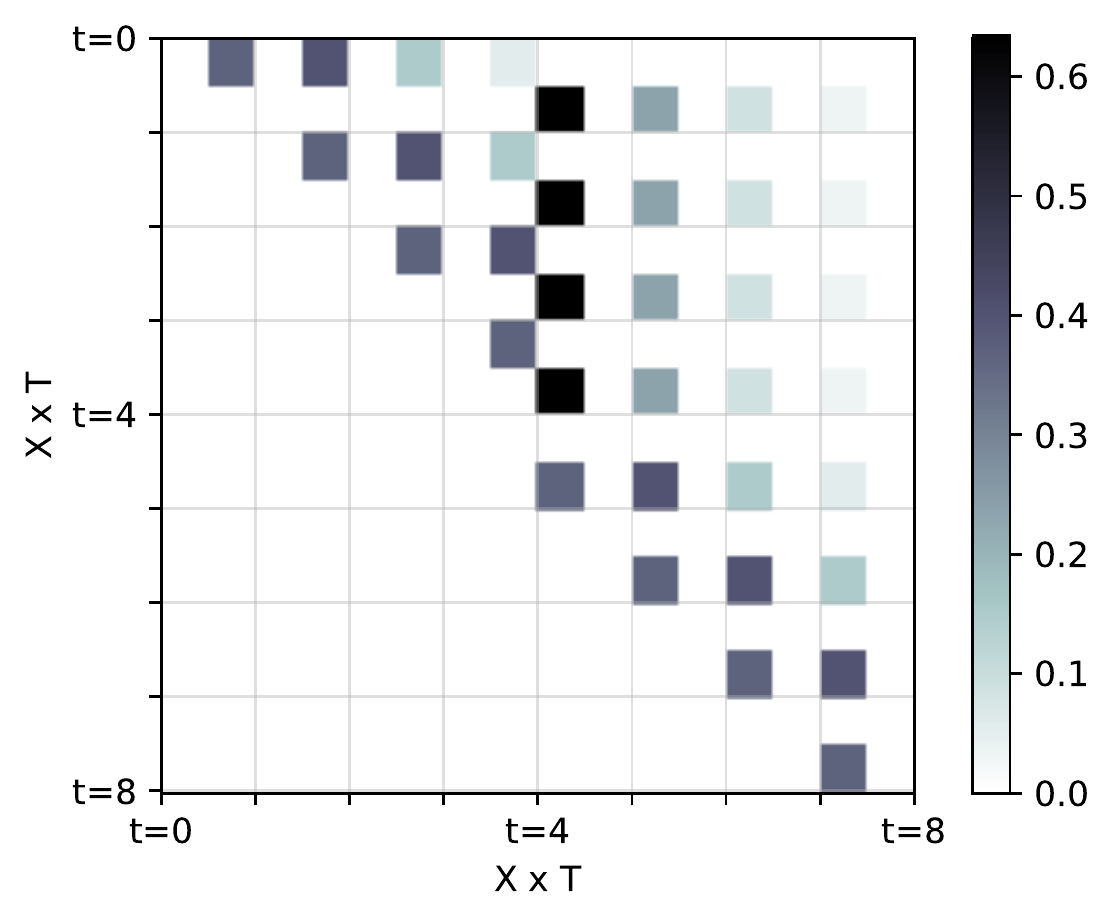}
    \end{subfigure}
    \begin{subfigure}{0.5\textwidth}
          \centering
          \includegraphics[width=\textwidth]{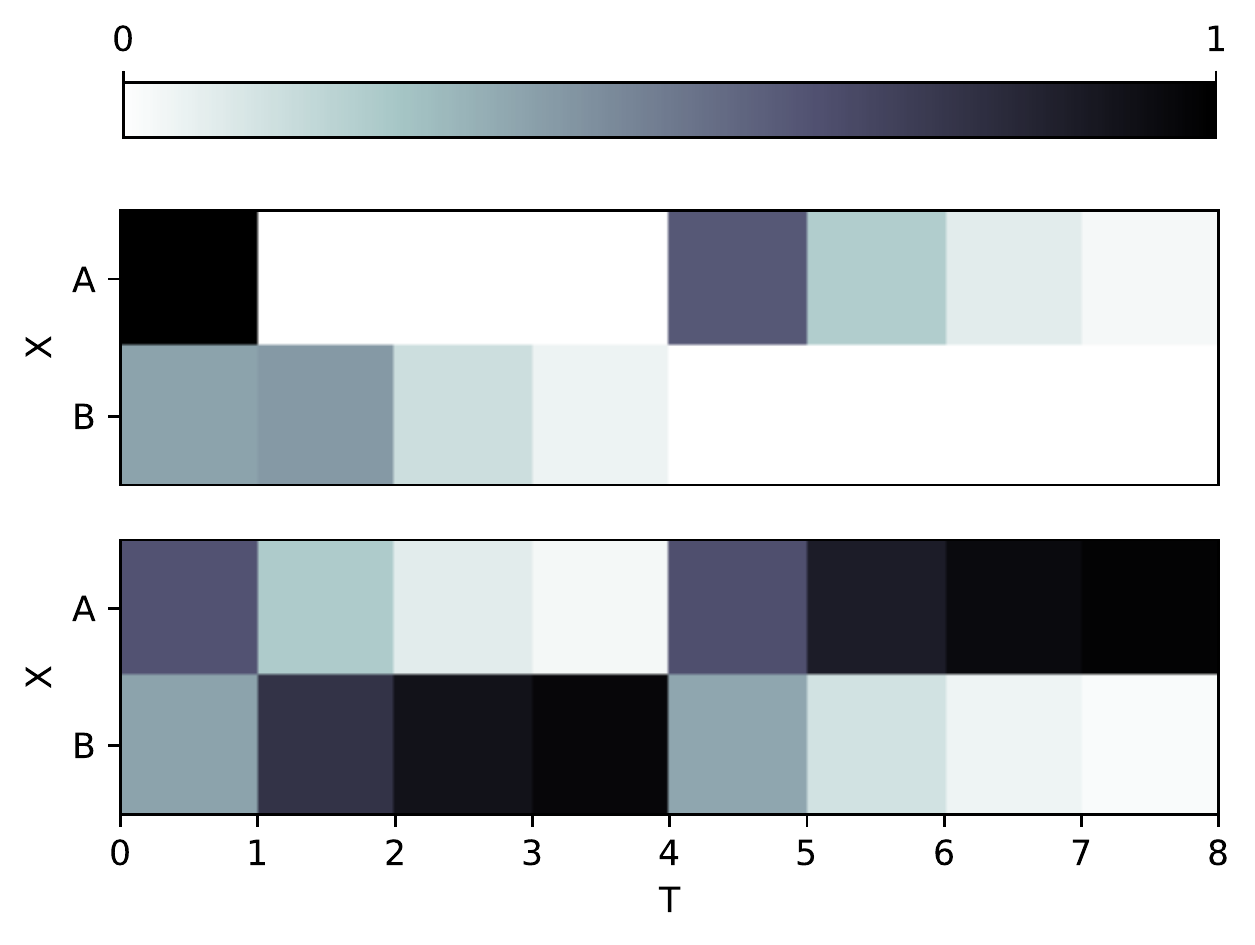}
        
    \end{subfigure}
    \caption{Illustration of the augmented jump chain for a 2 state system. \\
    Left: Discretized jump operator. Right: Jump activity (top) and recovered probability density (bottom) when the system starts in state $A$ at time 0.   } \label{fig:ex1}
\end{figure}
Since we used space-major ordering for space-time states each $2\times 2$ block represents the transitions from and into a time-slice whereas the position inside the blocks determines the spatial start- and end-positions (see also Figure \ref{fig:matrix}. Looking at the upper row of blocks we observe that for the initial two time blocks the dominant transitions are those from space-state A to B. This switches in the second half of the time interval, i.e. for the blocks on the right half of the matrix. That is, trajectories that started at time $0$ in state B will most likely jump after $t=4$, when $B$ is no longer absorbing.
We can also recognize the exponential decay of the probabilities with time.
The following rows of blocks encode the behaviour for the jumps starting from later times and mimic the qualitative behaviour of the top row although with different densities.

Starting from an initial distribution we are now in the position to look at the induced jump activity, its synchronization and the resulting Koopman operator.
Let us start with a space-time distribution $f \in \mathbb{R}^{2\times 8}$ with all mass in state $A$ at the initial time interval, i.e. $f_{x,t} = \delta_{x,A} \Ind_{t \in \{0,1\}}$. We then compute the jump activity from eq. \cref{eqn:jumpact} truncating the sum at $n=100$ for reasons of computability. The top right of Figure \ref{fig:ex1} depicts the resulting activity $Ef$ which can be understood as the amount of space-time jumps happening into each space-time cell akin to a Geiger counter.
We can identify the initial mass (top-left cell), as well as the intensity of the following jump destinations. The intensity decays with time since less and less particles remain available for the transition from $A\rightarrow B$ whereas the other direction is inhibited by the 0 rate. This changes at $t=4$ where we switch the reaction rates and observe a similar pattern in the reverse direction. The leftmost cells are a special case. Due to the discretization we don't start at time $t=0$ but uniformly in the first time-cell. Since the particles are spread out over that time interval the probability to jump is lower than the probability when starting at time $t=0$, which is why the activity in these cell is lower.

\subsection{Diffusion process with changing temperature}

In order to illustrate the applicability to molecular dynamics we now consider a diffusion process with drift induced by a potential. We reduce the temperature in time, akin to the process in simulated annealing. More precisely, we consider the overdamped Langevin equation in $\mathbb{R}^2$, 
\begin{equation*}
    \dd Y_t = - \nabla V(Y_t) \dd t + \sqrt{2 \beta(t)^{-1}} \dd W_t
\end{equation*}
with a triple well potential $V$ with 2 deep wells at $(-1,0)$, $(1,0)$ and a shallow well at $(0,\frac{3}{2})$ as in  \cite{helfmann2020extending}. $W_t$ denotes standard Brownian motion and $\beta(t)$ is the varying inverse of the temperature / the coldness.

We discretize the state-space on the domain $[-2,2] \times [-1,2]$ by dividing it into a square grid of $n_x=9$ horizontal and $n_y=7$ vertical points. In order to obtain the spatially discrete jump process approximation to the originally space-continuous process, we use the \emph{square-root approximation} (SQRA) \cite{donati2018sqra}.
The SQRA estimates a generator matrix on the space of states identified with the grid points by linearly  interpolating the potential between neighbouring points and calculating the resulting rates for a given temperature.
It is called SQRA since it can be expressed in terms of the square root of the Boltzmann weights as follows:
\begin{equation}
    Q_{ij} = \Phi A_{ij} \sqrt{\frac{\exp (-\beta V_j)}{\exp( -\beta V_i)}}
\end{equation}
where $A_{ij}$ denotes the the adjacency matrix of the grid points, $V_i$ the potential at grid point $i$ and the diagonal $Q_{ii}$ is set to satisfy row sum zero. The factor $\Phi$ amounts to the transition rate in a flat potential and depends on the $\beta$ as well as the spatial grid-size $h$ by $\Phi=\beta^{-1} h^{-2}$ \cite{donati2020}.

For the time domain we chose $T=[0,2]$ which we subdivide into $n_t=6$ uniform time cells of size $\Delta T=\frac{1}{3}$ and we impose an annealing protocol by starting with high temperature in the first half, $\beta(t)=1$ for $t \in [0,1)$, and decreasing it in the second half, $\beta(t)=10$ for $t \in [1,2]$.

\begin{figure}
    \centering
    \begin{subfigure}{0.49\textwidth}
            \label{fig:tw_kernel}
                  \centering
                  \includegraphics[width=\textwidth]{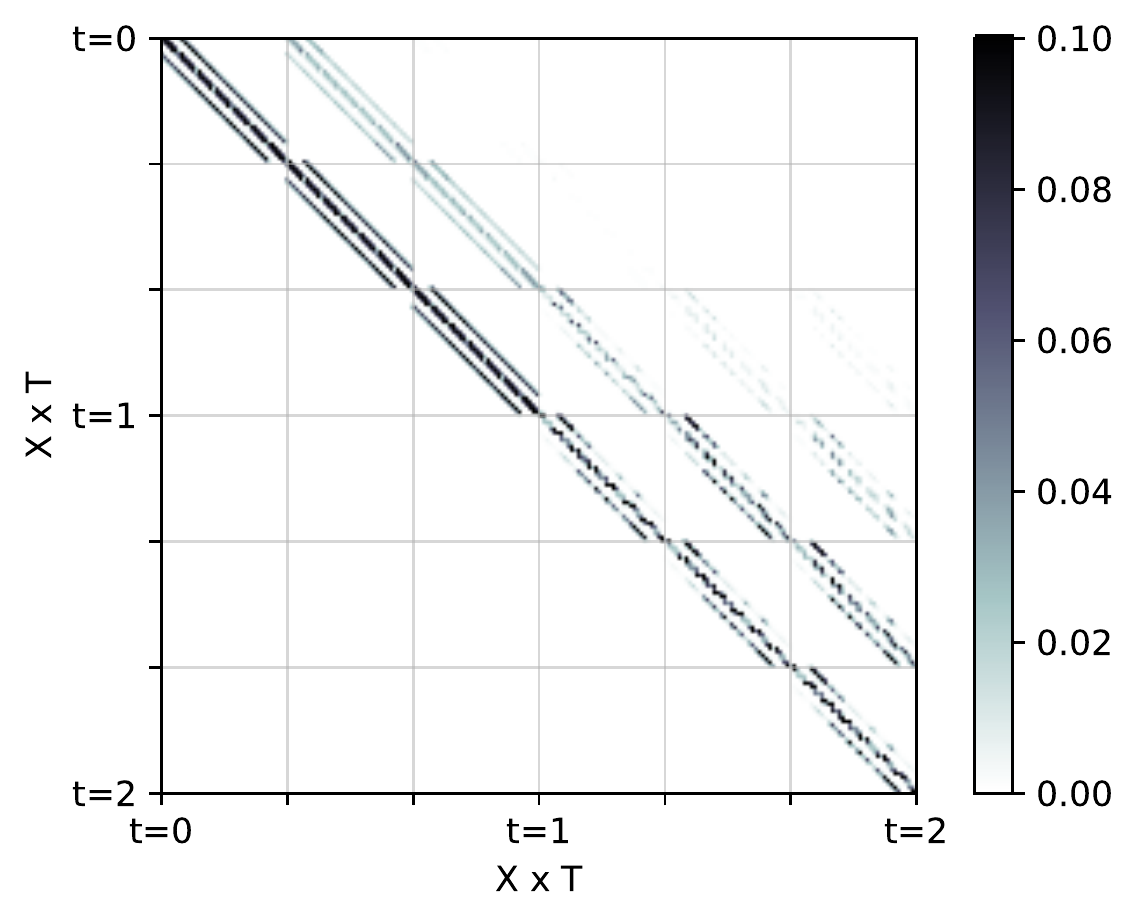}
    \end{subfigure}
    \begin{subfigure}{0.5\textwidth}
            \label{fig:tw_sparse}
          \centering
          \includegraphics[width=0.75\textwidth]{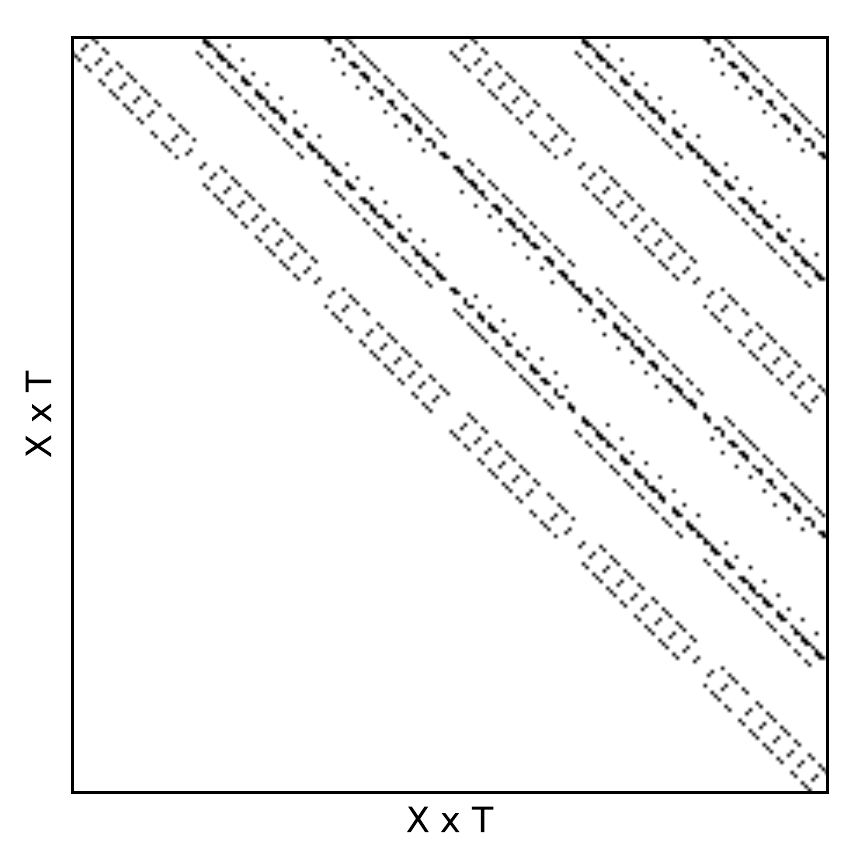}
        
    \end{subfigure}
    \caption{Two dimensional diffusion process with decreasing temperature. \\
    Left: Discretized jump operator. Right: Sparsity pattern.}
    \label{fig:ex2}
\end{figure}

The left of Figure \ref{fig:ex2} shows the corresponding discretization of the space-time jump operator $J$. We can recognize the high-temperature regime on the left half of the matrix by the rather uniform distribution of transition probabilities inside each block, as well as by the fast timescale of the reactions indicated by a high amount of temporal self-transitions on the diagonal blocks, with quickly decaying transitions to the future time blocks (almost none for the second off-diagonal).
On the other hand, the right half of the matrix encodes the behaviour of the low temperature-regime.
The distribution of transitions inside each block is more peaked as the potential-induced drift dominates the now small noise. We also see that the process slowed down since we have more transitions to the future blocks on the off-diagonal corresponding to particles that remain in place for longer times.

Whereas the matrix is $(n_x n_y n_t)^2 = 142884$ dimensional only $4620$ entries are nonzero, leading to a sparsity factor of $3.1\%$. The sparsity pattern is depicted on the right of Figure \ref{fig:ex2}.

The approximation error of the spatial discretization of the process by means of the SQRA is discussed in \cite{heida2018convergences}. 
We can analyze the approximation error $\epsilon$ of the temporal Galerkin approximation by comparing the reconstruction of the propagator $\MeasOp^t$ (\cref{sec:prop}\footnote{Here we approximate the discretized survival probabilities between the time-block as one minus the probabilities to leave the blocks, i.e. $\hat S_{ikl} = 1 - \sum_{j, s\le l} \hat J_{ikjs}  $ })
to the exact propagator $\prop^{0,t}$ of the Markov jump process obtained from the matrix exponential of Q (which is piecewise constant) by means of the $L^2$ operatornorm at the end-time $t=2$:
\begin{equation}
    \epsilon = \left\Vert \MeasOp^{2} - e^{Q(0)} e^{Q(1)} \right\Vert
\end{equation}
Figure \ref{fig:tw_error} shows the resulting error for our example for temporal step sizes between $0.01$ and $1$ and we observe convergence close to order 1.  
\begin{figure}[h]
\centering
\includegraphics[width=5cm]{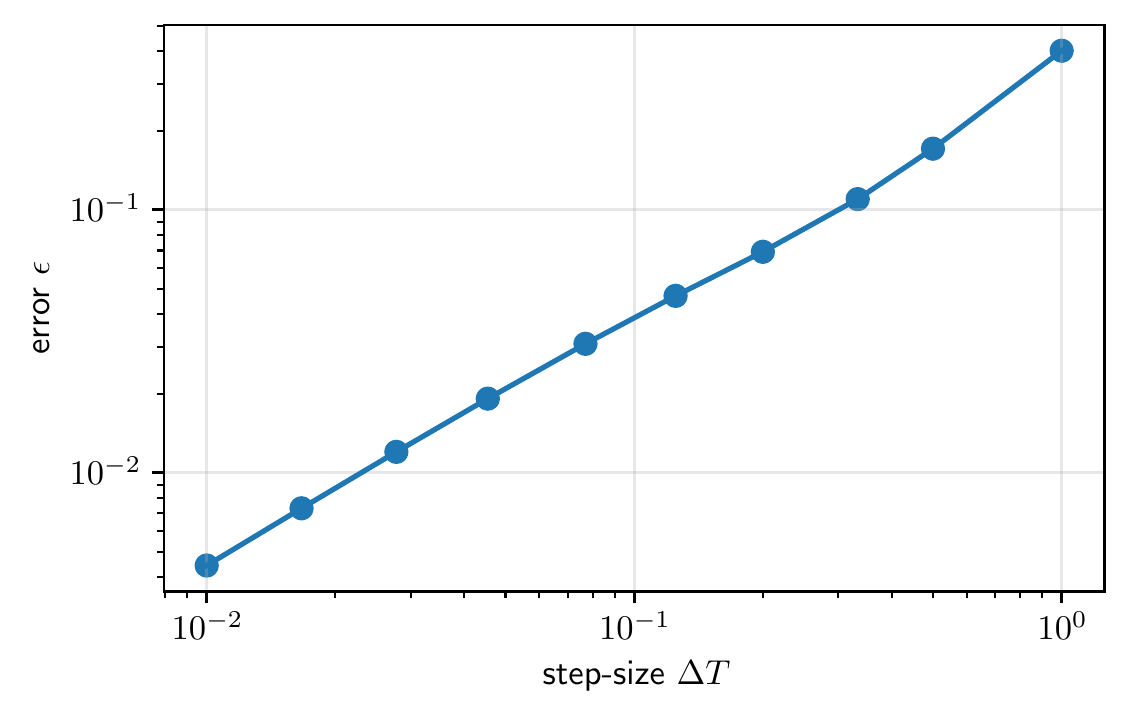}
\caption{Approximation error of the propagator reconstructed from the Galerkin approximation wrt. the temporal step size. }\label{fig:tw_error}
\end{figure}

Although these examples mainly serve to display the concept of the space-time augmented jump chain and merely recompute already known quantities, they also illustrate its main strength, i.e. dealing with non-autonomous processes in a sparse way whilst encoding structural properties such as the mixing behaviour and timescales of the underlying problem.

\section{Conclusion}
We extended the known representation of autonomous Markov jump processes as embedded Markov chain (\cref{thm:aut}) to the non\hyp{}autonomous regime.
Augmenting the state space with the time dimension allows us to encode the temporal dependence of the embedded chain in the new space-time state space. Therefore we end up with an time\hyp{}independent representation for the system. While the augmentation is a common technique for non\hyp{}autonomous systems, the novelty of our approach is that we only look at the jump events themselves. This allows us to move from a non\hyp{}autonomous continuous-time Markov process to an autonomous discrete-time Markov chain (\cref{thm:ajc}), albeit on a more complex state space.
We call this Markov chain the \emph{augmented jump chain} and characterize it through its transition kernel and evolution operator, the \emph{jump operator}.

This approach leads to a fundamentally new perspective on time: Whereas classically time progresses uniformly, we now have a description where the process jumps through time concurrently.
Whilst it is possible to revert to the classical picture through a synchronization, i.e. by assigning a membership along each space-fibre of the augmented system towards a specific time\hyp{}point in uniform time, it is interesting to see that many problems can be tackled in the augmented regime directly.
We showed how the evaluation of the Koopman operator, i.e. the evolution of an observable through time, can be solved directly in the "desynchronized" regime in the form of an inhomogeneous linear boundary value problem on space-time (\cref{cor:koop}).
This problem structurally resembles the one for the computation of committor functions in stationary systems.

We discuss connections of our representation to the computation of committors for time\hyp{}independent target sets but non\hyp{}autonomous dynamics.
The time\hyp{}augmented perspective furthermore allows for a natural extension to a wide class of time\hyp{}dependent targets and eventually a non\hyp{}autonomous committor theory.
We furthermore discuss the application of the augmentation to the theory of coherence where it seems to provide a promising view on capturing time\hyp{}invariant structures.

The defining principles of our proposed approach are twofold. For one the well-known technique of augmentation allows us to treat non\hyp{}autonomous system and extend common notions of analysis (committors, metastability) in a unifying way to the time\hyp{}dependent regime.
The other however is far less understood: By focusing on the jump events as main principle of evolution in contrast to the usual focus on time, we arrive at a description where the classical time evolves concurrently. We show how this leads to a representation inheriting the sparsity of the infinitesimal generator. This in itself may prove to be very useful for the computational analysis of (especially high-dimensional) non\hyp{}autonomous systems.
However interpreting the concurrency as uncoupling of different time\hyp{}scales requires further research and we believe that it becomes a cornerstone for the analysis of complex dynamics with multiple-timescales.

All in all, we hope for the \emph{augmented jump chain} to enhance the numerical capabilities for complex systems on the applied side as well as opening doors to new perspectives for time\hyp{}dependent jump processes on the theoretical side.

\paragraph{Acknowledgments}
We would like to thank Luzie Helfmann for many insightful discussions, especially about committor functions, as well as proof-reading.
This research has been funded by Deutsche Forschungsgemeinschaft (DFG) through grant CRC 1114 "Scaling Cascades in Complex Systems", Project Number 235221301.
\section*{Appendix}
\subsection*{The non-homogeneous exponential distribution} \label{sec:nonhomexp}

Albeit what we call the non-homogeneous exponential distribution may very likely be already known, e.g. in the field of survival analysis, we could not find any published references. We therefore present a short derivation based on an answer on stackexchange \cite{2573322}:

Define the non-homogeneous exponential distribution (NED) with rate $q: \mathbb{R}^+ \rightarrow \mathbb{R}^+$ by the cumulative distribution function (CDF)
$$\PP[t>T] = F(t) = 1-\exp\left(-\int_{0}^{t} q(s) \dd s\right),$$
where $T$ is the NED distributed random time.
Note that $F$ indeed is a CDF:
$$F(0) = 0, \lim_{t\rightarrow\infty} F(t) = 1.$$
Then its derivative is the probability distribution function (PDF)
$$f(t) = \frac{dF}{dt}(t) = q(t) \exp\left(-\int_{0}^{t} q(s) \dd s\right).$$
Now consider the conditional probability
$$p_{\Delta t} (t) = \PP(t+\Delta t > T \mid T>t)  = \frac{F(t+\Delta t) - F(t)}{1-F(t)}$$
and its rate, i.e. the limit for $\Delta t \rightarrow 0$
$$\lambda(t) = \lim_{\Delta t \rightarrow 0} \frac{p_{\Delta t}}{\Delta t} (t) = \frac{F'(t)}{1-F(t)} = \frac{f(t)}{1-F(t)} = q(t)$$
The homogeneous exponential distribution (HED) with rate $q$ is a special case of the NED with $q\equiv q(t)$.
Hence the NED has the same conditial rate as the HED for an event occuring at each time $t$, i.e. is the consistent generalization to non\hyp{}autonomous rates.

Furthermore the survival probability satisfies
$$ S(t) = \PP [t<T] = 1-F(t) = \exp\left(-\int_{0}^{t} q(s) \dd s\right). $$

\printbibliography

\end{document}